\documentclass{amsart}
\usepackage{amsmath,amssymb,amsthm}

\theoremstyle{plain}
\newtheorem{theorem}{Theorem}[section]
\newtheorem{lemma}[theorem]{Lemma}

\newtheorem{proposition}[theorem]{Proposition}

\theoremstyle{definition}

\newtheorem{remark}[theorem]{Remark}

 \newcommand{\al}{\alpha}

\def\bN{\hbox{$\mathbb N$}}
\def\bZ{\hbox{$\mathbb{Z}$}}

\def\tensor{\hbox{$\widehat{\otimes}$}}

\def\ho{\hbox{$\hat\omega$}}

\def\L1o{\hbox{$L^1(G,\omega)$}}
\def\l1o{\hbox{$\ell^1(G,\omega)$}}
\def\Ld1o{\hbox{$L^1(G\times G, {\omega\times \omega})$}}
\def\ld1o{\hbox{$\ell^1(G\times G, {\omega\times \omega})$}}
\def\LH1o{\hbox{$L^1(G/H, {\ho})$}}

 \def\C0o{\hbox{$C_0(1/\omega)$}}
 \def\Lio{\hbox{$L^\infty\bigl(G,\frac1{\omega}\bigr)$}}

\newcommand{\R}{\mathbb{R}}

\newcommand{\eps}{\varepsilon}

 \renewcommand{\leq}{\leqslant}
\renewcommand{\geq}{\geqslant}

\newcounter{equi1}

\newcommand{\bea}{\begin{eqnarray*}}
\newcommand{\eea}{\end{eqnarray*}}
\newcommand{\beq}{\begin{equation}}
\newcommand{\eeq}{\end{equation}}
\newcommand{\begsta}{\begin{statements}}
\def\endsta{\end{statements}}
\newcommand{\begaeq}{\begin{aequivalenz}}
\def\endaeq{\end{aequivalenz}}

\begin{document}
\title[Central Beurling algebras]{Weak amenability of the central Beurling algebras on [FC]$^-$ groups}



\author[V. Shepelska]{Varvara Shepelska\dag}
\email{\dag\; shepelska@gmail.com}

\author[Y. Zhang]{ Yong Zhang \ddag}
\email{\ddag\; zhangy@cc.umanitoba.ca}
\thanks{\ddag \; Supported by NSERC Grant 238949}
\address{Department of Mathematics\\
           University of Manitoba\\
           Winnipeg, Manitoba\\
           R3T 2N2 Canada}

\date{November 24, 2014}

\begin{abstract}
We study weak amenability of central Beurling algebras $ZL^1(G,\omega)$. The investigation is a natural extension of  the known work on the commutative Beurling algebra $L^1(G,\omega)$. For [FC]$^-$ groups we establish a necessary condition and for [FD]$^-$ groups we give sufficient conditions for the weak amenability of $Z\L1o$. For a compactly generated [FC]$^-$ group with the polynomial weight $\omega_\alpha(x) = (1 + |x|)^\alpha$, we prove that $ZL^1(G,\omega_\alpha)$ is weakly amenable if and only if $\alpha < 1/2$.
\end{abstract}

\maketitle

\section{Introduction}\label{Intro}

Let $G$ be a locally compact group. As it is customary, two functions equal to each other almost everywhere on $G$ with respect to the Haar measure will be regarded as the same. We denote the integral of a function $f$ on a (Borel) measurable subset $K$ of $G$ against a fixed left Haar measure by $\int_Kfdx$. The space of all complex valued Haar integrable functions on $G$ is denoted by $L^1(G)$. 
A weight on $G$ is a Borel measurable function $\omega:G\to\mathbb{R}^{+}$ satisfying 
\[
\omega(xy)\le\omega(x)\omega(y) \quad (x,y\in G).
\]
Given a weight $\omega$ on $G$, we consider the space $\L1o$ of all complex valued Haar measurable functions $f$ on $G$ that satisfy
\[
\|f\|_\omega = \int{|f(x)|\omega(x)dx} < \infty.
\]
With the convolution product $*$ and the norm $\|\cdot\|_\omega$, $\L1o$ is a Banach algebra, called a Beurling algebra on $G$. When $\omega = 1$, this is simply the group algebra $L^1(G)$. Let $Z\L1o$ be the closed subalgebra of $\L1o$ consisting of all $f\in \L1o$ such that $f^g=f$ for all $g\in G$, where $f^g(x) = f(g^{-1}xg)$ ($x\in G$). Then $Z\L1o$ is a commutative Banach algebra, called a \emph{central Beurling algebra} \cite{LM}. Indeed, $Z\L1o$ is the center of $\L1o$. It is well known that $Z\L1o$ is non-trivial if and only if $G$ is an [IN] group \cite{Mosak}. 

From 
\cite[Remark~8.8]{GLZ}, a measurable weight $\omega$ on $G$ is always equivalent to a continuous weight $\tilde\omega$ on $G$, where the equivalence means that there are constants $c_1, c_2 >0$ such that $c_1\omega(x) \leq \tilde\omega(x) \leq c_2\omega(x)$ for almost all $x\in G$. The equivalence implies that the respective Beurling algebras $\L1o$ and $L^1(G,\tilde\omega)$ are isomorphic as Banach algebras. So are the central Beurling algebras $Z\L1o$ and $ZL^1(G, \tilde\omega)$. For this reason, in our investigation we will most time assume the weight $\omega$ to be a continuous function.

The study of $\L1o$ goes back to A. Beurling \cite{Beurling1938}, where $G=\mathbb{R}$ was considered. One may find a good account of elementary theory concerning the general weighted group algebra in \cite{RS}. Structure of the center of group algebras and the central Beurling algebras were substantially studied in \cite{LM, LM_center}.

Amenable Beurling algebras are essentially isomorphic to amenable group algebras \cite{Gronbaek, White}. This is no longer true for the weak amenability of them. Weak amenability of Beurling algebras for commutative groups has been extensively investigated and are well characterized \cite{BCD, Gronbaek2, Samei, Zhang}, while  for non-commutative groups, one hardly sees a non-trivial example of a weakly amenable Beurling algebra. Recent investigations are in \cite{Borwick, Shepelska, SZ1}.

We are concerned with weak amenability of central Beurling algebras $Z\L1o$. When $G$ is commutative, $Z\L1o$ coincides with $\L1o$. So investigation of weak amenability for $Z\L1o$ is a natural extension of the study for the commutative groups. We notice that some investigations on amenability and weak amenability of $ZL^1(G)$, the case with the trivial weight $\omega =1$, have been made recently in \cite{ASSp} and \cite{ A-C-S}. All theses studies even for amenability have answers only for particular cases of locally compact groups, in particular the ones in \cite{ASSp} and \cite{ A-C-S} are only on compact and some discrete groups.

We will mainly focus on [FC]$^-$ groups (groups with precompact conjugacy classes). 
We note that for the trivial weight $\omega =1$ it has been shown in \cite{ASSp} that $ZL^1(G)$ is always weakly amenable for an [FC]$^-$ group $G$. When $G$ is compact, the result has a simple direct proof.  In fact, $ZL^1(G)$ is generated by its minimal idempotents if $G$ is compact. So by a simple observation (see \cite[Section~7]{Johnson}), $ZL^1(G, \omega)\simeq ZL^1(G)$ is always weakly amenable for compact $G$. 

We recall here some standard notions concerning a locally compact group. More details regarding them may be found in \cite{GM2, GM, Palmer}.

%
%

For a locally compact group $G$,
\begin{enumerate}
\item $G$ is an [IN] group if there is a compact neighbourhood of the identity that is invariant under all inner automorphisms;
\item $G$ is a [SIN] group if there is a compact neighbourhood basis of the identity such that each member of the basis is invariant under inner automorphisms;
\item $G$ is an [FC]$^-$ group if the conjugacy class $\{gxg^{-1}: g\in G\}$ for each $x\in G$ has a compact closure in $G$;
\item $G$ is an [FD]$^-$ group if the closure of the commutator subgroup $G'$ of $G$ is compact in $G$ (where the commutator subgroup of $G$ is the subgroup generated by all elements of the form $xyx^{-1}y^{-1}$, $x,y\in G$).
\end{enumerate}

Obviously, $G$ is an [FD]$^-$ group if and only if there is a compact normal subgroup $K$ such that $G/K$ is abelian.
It is also obvious that [FD]$^-\subseteq$ [FC]$^-$. The less obvious inclusion [FC]$^-\subseteq$ [IN] was shown in \cite[Proposition~3.1]{Liukkonen}. 

In Section~2 we will be devoted to considering [FC]$^-$ groups. We will show that the projective tensor product $ZL^1(G_1,\omega_1)\tensor ZL^1(G_2,\omega_2)$ is weakly amenable if and only if both $ZL^1(G_1,\omega_1)$ and $ZL^1(G_2,\omega_2)$ are weakly amenable and, under some conditions, $ZL^1(G_1\times G_2,\omega_1\times\omega_2)\simeq ZL^1(G_1,\omega_1)\tensor ZL^1(G_2,\omega_2)$. We will also show, among others, that a condition characterizing the weak amenability of the Beurling algebra on a commutative group remains necessary for central Beurling algebras on [FC]$^-$ groups. In section~3 we will focus on the case for [FD]$^-$ groups, establishing some sufficient conditions for $Z\L1o$ to be weakly amenable for this sort of groups $G$. For a compactly generated [FC]$^-$ group $G$, we will consider the polynomial weight $\omega_\alpha(x) = (1+|x|)^\alpha$. We will show that $ZL^1(G,\omega_\alpha)$ is weakly amenable if and only if $0\leq \alpha < 1/2$. This last result notably generalizes \cite[Theorem~2.4.(iii)-(iv)]{BCD}.

\section{Central Beurling algebras on [FC]$^-$ groups}

Let $G$ be a locally compact group, and let $Aut(G)$ be the set of all topological automorphisms of $G$ onto itself.   For any compact set $F$ of $G$ and any open neighbourhood $U$ of $e$ in $G$ we denote
\[
B(F,U) = \{\tau\in Aut(G): \tau(x)\in Ux, \tau^{-1}(x)\in Ux \text{ for all } x\in F\}.
\]
The family of all subsets of the form $B(F,U)$ forms an open neighbourhood basis at the identity $\imath$ of  $Aut(G)$.
This defines a Hausdorff topology on $Aut(G)$, called the open compact topology on $Aut(G)$. It is well-known that this topology is completely regular \cite[Theorems~4.8 and 8.4]{H-R}. With this topology $Aut(G)$ is a topological group (but it may not be locally compact) \cite[Theorems~26.5 and 26.6]{H-R}.  All inner automorphisms of $G$ form a (completely regular) topological subgroup of $Aut(G)$, denoted by $I(G)$. For $x\in G$, let $\beta_x$ be the inner automorphism of $G$ implemented by $x$, i.e.
 $$\beta_x(g) = xgx^{-1} \quad (g\in G). $$
Clearly, the natural mapping $x\mapsto  \beta_x$ is a continuous group anti-homomorphism from $G$ onto $I(G)$, so that
$x\mapsto (\beta_x)^{-1}= \beta_{x^{-1}}$ is a continuous group homomorphism from $G$ onto $I(G)$. 

Let $S$ be a semitopological semigroup, i.e. $S$ is a semigroup with a Hausdorff topology such that the product $s\cdot t$ is separately continuous. Obviously, a topological group is a semitopological semigroup, in particular, $I(G)$ belongs to this class.
The space $C_b(S)$ of all bounded complex-valued continuous functions on $S$ forms a Banach space with the supremum norm 
\[
\|f\|_{\sup} = \sup_{s\in S}|f(s)| \quad (f\in C_b(S)).
\]
Indeed, with the pointwise product $C_b(S)$ is a unital commutative C*-algebra whose identity is the constant function $1$.
Let $s\in S$ and $f\in C_b(S)$. The \emph{left translate} $\ell_sf$ of $f$ by $s$ is the function $\ell_sf(x) = f(sx)$ ($x\in S$). The {\em right translate} $r_sf$ by $s$ is defined similarly. Clearly for each $s\in S$, $\ell_s$ and $r_s$ define bounded operators on $C_b(S)$. 
A positive linear functional $m\in C_b(S)^*$ is called a \emph{left invariant mean} on $C_b(S)$ if $\|m\| = m(1) = 1$ (i.e. $m$ is a mean) and $m(\ell_sf) =m(f)$ (i.e. $m$ is left invariant) for all $f\in C_b(S)$ and all $s\in S$. Similarly, a \emph{right invariant mean} on $C_b(S)$ is a a mean $m\in C_b(S)^*$ satisfying $m(r_sf) =m(f)$ 
for all $f\in C_b(S)$ and all $s\in S$. For a locally compact group $G$, it is well-known that $G$ is amenable if and only if $C_b(G)$ has an invariant mean, a mean on $C_b(G)$ which is both left invariant and right invariant \cite{greenleaf}.

For two semitopological semigroups $S$ and $H$, it is readily seen that if $C_b(S)$ has a left invariant mean and if there is a continuous semigroup homomorphism from $S$ onto $H$, then $C_b(H)$ has a left invariant mean. It is also readily seen that if there is a continuous anti-homomorphism from $S$ onto $H$, then the existence of a right invariant mean on $C_b(S)$ implies the existence of a left invariant mean on $C_b(H)$.

Let $G$ be a locally compact group, and let $f\in C_b(G)$. Then, for each $x\in G$ we have $\hat f_x \in C_b(I(G))$, where $\hat f_x(\beta) = f(\beta^{-1}(x))$ ($\beta\in I(G)$).
Suppose further that $G$ is an [FC]$^-$ group. Then, as is well-known, $G$ is amenable (see \cite{Palmer} or \cite{Leptin}). The above discussion shows that  $C_b(I(G))$ has a left invariant mean, say $\mu\in C_b(I(G))^*$. 
Note that $C_b(I(G))^* = M(\delta I(G))$, where $\delta I(G)$ is the Stone-C\v{e}ch compactification of $I(G)$ \cite[Corollary~V.6.4]{Conway}. 
 Let $f\in C_{00}(G)$ and $K = \text{supp}(f)$. Then
 \[
 C_K = cl\{\beta(x): \beta\in I(G), x\in K \}
 \]
 is a compact subset of $G$
 and the function $(\beta,x)\mapsto f(\beta^{-1}(x))$ is a continuous function on $I(G)\times G$ whose support sits in $I(G)\times C_K$. As explained above, we may regard the left invariant mean $\mu$ as in $M(\delta I(G))$. Restricting $\mu$ to $I(G)$, we obtain a positive finite Borel (probability) measure space $(I(G), \mu)$. Note that $C_K$ is of finite Haar measure as a compact subset of $G$. We then may apply Fubini's Theorem to the function $(\beta,x)\mapsto f(\beta^{-1}(x))$ on $I(G)\times C_K$ and define $P_\mu(f)$  by
\begin{equation}\label{projection}
P_\mu(f)(x) = \mu(\hat f_x) \quad (x\in G).
\end{equation}
Clearly $\text{supp}(P_\mu(f)) \subset C_K$ and $P_\mu(f)\in L^1(G)$. By the left invariance of $\mu$ it is readily seen that $P_\mu(f)(\beta^{-1}(x))=P_\mu(f)(x)$ ($\beta\in I(G)$ and $x\in G$). Whence $P_\mu(f) \in ZL^1(G)$. Moreover, the Fubini Theorem (\cite[Theorem~13.8]{H-R}) asserts that
\[
\|P_\mu(f)\|_1 \leq \|f\|_1.
\]
Since $C_{00}(G)$ is dense in $L^1(G)$, $P_\mu$ extends to a bounded linear operator from $L^1(G)$ into $ZL^1(G)$, still denoted by $P_\mu$. It is also easily seen that $P_\mu(f) = f$ when $f\in ZL^1(G)$. Therefore $P_\mu$: $L^1(G)\to ZL^1(G)$ is a Banach space contractive projection. Although $P_\mu$ is usually not a Banach algebra homomorphism, it is a $ZL^1(G)$-bimodule morphism if we view both $L^1(G)$ and $ZL^1(G)$ as natural $ZL^1(G)$-bimodules. But we will not use this feature. 

If $\mu_i$ is a left invariant mean on $C_b(I(G_i))$ ($i=1,2$), then $\mu_1\times \mu_2$ is a left invariant mean on $C_b(I(G_1)\times I(G_2))$. Note $I(G_1) \times I(G_2) = I(G_1\times G_2)$. This generates a contractive projection $P_{\mu_1\times \mu_2}$ from $L^1(G_1\times G_2)$ onto $ZL^1(G_1\times G_2)$. On the other hand, the mapping $f\otimes g \mapsto f\times g$ ($f\in L^1(G_1), g\in L^1(G_2)$) defines  a Banach algebra isometry $T$: $L^1(G_1) \tensor L^1(G_2) \to L^1(G_1\times G_2)$ which maps $ZL^1(G_1)\otimes ZL^1(G_2)$ into $ZL^1(G_1\times G_2)$, where $f\times g$ denotes the function
$$f\times g(x,y) = f(x)g(y) \quad (x\in G_1, y\in G_2).$$
Since $ZL^1(G_i)$ is complemented in $L^1(G_i)$, the inclusion mappings 
$$\imath_i: ZL^1(G_1)\to L^1(G_i)\quad  (i=1,2)$$
 induce a norm preserving Banach algebra embedding:
$$
\imath_1 \otimes \imath_2: ZL^1(G_1) \tensor ZL^1(G_2) \to L^1(G_1) \tensor L^1(G_2).
$$
We warn here that, in general for closed subspaces $B_i$ of Banach spaces $A_i$ and embeddings $\imath_i$: $B_i\to A_i$ ($i=1,2$), $\imath_1\otimes \imath_2$: $B_1\tensor B_2\to A_1\tensor A_2$ is not necessarily an embedding; it may not be even injective (see  \cite{Zhang_1999}).
Denote the inclusion mapping $ZL^1(G_1\times G_2)\to L^1(G_1\times G_2)$ by $\imath$. Then one can easily verify that 
$$
 \imath\circ P_{\mu_1\times \mu_2}\circ T = T\circ \imath_1\otimes \imath_2\circ P_{\mu_1}\times P_{\mu_2}.
$$


\begin{lemma}\label{centers}
Let $G_1$ and $G_2$ be locally compact [FC]$^-$ groups. Then, as Banach algebras,
$$
ZL^1(G_1\times G_2)\simeq ZL^1(G_1)\hat{\otimes}ZL^1(G_2).
$$
\end{lemma}
\proof

Consider the following chain:
\[
ZL^1(G_1) \tensor ZL^1(G_2) \overset{\imath_1\otimes\imath_2}{\longrightarrow} L^1(G_1) \tensor L^1(G_2)\overset{T}{\to} L^1(G_1\times G_2) \overset{P_{\mu_1\times\mu_2}}{\longrightarrow} ZL^1(G_1\times G_2).
\]
From the above discussion the composition of the chain provides a Banach algebra isomorphism (in fact, isometric isomorphism) from $ZL^1(G_1) \tensor ZL^1(G_2)$ onto $ZL^1(G_1\times G_2)$. 
\qed

Here we note that Lemma~\ref{centers} generalizes former known results for compact and discrete cases in \cite{ASSp}  and \cite{A-C-S}.

Now we consider the weighted case. If $\omega$ is a continuous weight on the [FC]$^-$ group $G$, then for each $x\in G$ there is a constant $c_x\geq 1$ such that 
$$\omega(\beta (x)) \leq c_x \omega(x)\quad  (\beta\in I(G)).$$
 Assume that there is an upper bound for all $c_x$. We then have the following.


\begin{proposition}\label{weightedcenters}
Let $G_1$, $G_2$ be [FC]$^-$ groups, and $\omega_i$ be a weight on $G_i$ satisfying $\omega_i(gxg^{-1})\le c\,\omega_i(x)$ ($x,g\in G_i$) ($i=1,2$), where $c>0$ is a constant. Then, as Banach algebras,
\begin{equation}\label{tensor centers}
ZL^1(G_1\times G_2,\omega_1\times\omega_2)\simeq ZL^1(G_1,\omega_1)\hat{\otimes}ZL^1(G_2,\omega_2).
\end{equation}
\end{proposition}

\begin{proof} Up to equivalence we may assume $\omega_1$ and $\omega_2$ to be continuous.
 If $\omega$ is a continuous weight on an [FC]$^-$ group $G$ such that $\omega(gxg^{-1})\le c\,\omega(x)$ ($x,g\in G$), we can still consider the mapping $P_\mu$ on $C_{00}(G)$ defined by (\ref{projection}). Let $f\in C_{00}(G)$. Then we have
\[
|P_\mu(f)(x)\omega(x)| \leq c P_\mu (|f|\omega)(x), \quad |P_\mu (|f|\omega)(x)| \leq c P_\mu(|f|)(x)\omega(x) \quad (x\in G).
\]
By the Fubini Theorem we obtain
\[
\|P_\mu(f)\|_\omega \leq c \|f\|_\omega
.
\]
These are true for all $f\in C_{00}(G)$, which is dense in $\L1o$. So $P_\mu$ extends to a bounded linear mapping from $\L1o$ to $\L1o$. Similar to the non-weighted case, we have $P_\mu(f)\in ZL^1(G, \omega)$. So $P_\mu$ is a continuous projection from $\L1o$ onto $ZL^1(G, \omega)$ and $\|P_\mu\|\leq c$. Then we follow the same argument for Lemma~\ref{centers} to get the isomorphic relation (\ref{tensor centers}).
\end{proof}

As  is well-known, [FC]$^-$ groups are amenable [IN] groups \cite{Palmer}. For general amenable [IN] groups we have the following result.
\begin{proposition}\label{weakly amenable tensor}
Let $G_1$ and $G_2$ be amenable [IN] groups, and let $\omega_1$ and $\omega_2$ be  weights on them, respectively.
Then $ZL^1(G_1,\omega_1)\hat{\otimes}ZL^1(G_2,\omega_2)$ is weakly amenable if and only if both $ZL^1(G_1,\omega_1)$ and $ZL^1(G_2,\omega_2)$ are weakly amenable.
\end{proposition}

\proof Again, we may assume that the weights are continuous.

Since $ZL^1(G_1,\omega_1)$ and $ZL^1(G_2,\omega_2)$ are commutative, The sufficiency follows from \cite[Proposition~2.8.71]{Dales}. 

For the converse, we first note that, if $G$ is an amenable [SIN] group and $\omega$ is a weight on $G$, then there is a character (namely, a bounded multiplicative linear functional) $\varphi>0$ on $L^1(G,\omega)$ from \cite[Lemma~1]{White}. Restricting to $ZL^1(G,\omega)$, $\varphi$ is clearly non-trivial (note that $L^1(G, \omega)$ has a central bounded approximate identity). Now let $G$ be an amenable [IN] group. Then it is well-known that there is a compact normal subgroup $K$ of $G$ such that $G/K \in$[SIN] (see \cite[Theorem~1]{Iwosawa}) and $G/K$ is still amenable. Define a weight $\hat\omega$ on $G/K$ by
\[
\hat\omega(xK) = \inf_{t\in K}\omega(xt).
\]
Then there is a standard Banach algebra homomorphism $T$ from $\L1o$ onto $L^1(G/K, \hat\omega)$ (\cite[Theorem~3.7.13]{RS}). In fact, $T$ is precisely formulated by
\[
T(f)(xK) = \int_K f(xt)dt \quad (f\in \L1o, x\in G).
\]
Clearly, $T$ maps $Z\L1o$ onto $ZL^1(G/K, \hat\omega)$.
As we have shown, there is a character $\varphi$ on $L^1(G/K,\hat\omega)$ which is nontrivial on $ZL^1(G/K,\hat\omega)$. Then the composition $\phi = \varphi\circ T$ gives a character on $\L1o$ which is nontrivial on $Z\L1o$. Apply this to $ZL^1(G_2,\omega_2)$ and 
assume $ZL^1(G_1,\omega_1)\hat{\otimes}ZL^1(G_2,\omega_2)$ is weakly amenable. Then the mapping 
\[
f\otimes g \mapsto \phi(g) f \quad (f\in ZL^1(G_1,\omega_1), g\in ZL^1(G_2, \omega_2))
\]
generates a Banach algebra homomorphism from $ZL^1(G_1,\omega_1)\hat{\otimes}ZL^1(G_2,\omega_2)$ onto $ZL^1(G_1,\omega_1)$. Whence $ZL^1(G_1,\omega_1)$ is weakly amenable by  \cite[Proposition~2.8.64]{Dales}. Similarly,
$ZL^1(G_2,\omega_2)$ is weakly amenable.
\qed

Consider the special case $G=H\times K$, where $H$ is an [FC]$^-$ group and $K$ is a compact group. Let $\omega$ be a continuous weight on $G$. Define 
$$\omega_H(x) = \omega(x,e_K)\quad (x\in H), $$
where $e_K$ is the unit of $K$. Then $\omega_H$ is a weight on $H$, and $\omega$ is equivalent to the weight $\omega_H\times 1$ on $H\times K$. Therefore,
$ZL^1(G, \omega)$ is a Banach algebra isomorphic to $ZL^1(G,\omega_H\times 1)$. Now assume $\omega$ satisfies $\omega(hxh^{-1}, e_K) \leq c\omega(x,e_K)$ ($x, h \in H$) for some constant $c$.
Since $ZL^1(K)$ is weakly amenable (see \cite[Proposition~5.1]{Zhang}), by Propositions~\ref{weightedcenters} and \ref{weakly amenable tensor} we see $Z\L1o$ is weakly amenable if and only if $ZL^1(H,\omega_H)$ is weakly amenable. This, in particular, leads us to the following extension of \cite[Theorem~3.1]{Zhang}, where $ZL^1(H,\omega_H)=L^1(H,\omega_H)$ since $H$ is abelian.

\begin{proposition}\label{abelian/compact}
Suppose that $G=H\times K$, $H$ is an abelian group and $K$ is a compact group. Let $\omega$ be a weight on $G$. Then $ZL^1(G,\omega)$ is weakly amenable if and only if there is no non-trivial continuous group homomorphism $\Phi:G\to\mathbb{C}$ such that
\begin{equation}\label{32}
\sup\limits_{g\in\,G}\ \frac{|\Phi(g)|}{{\omega}(g){\omega}(g^{-1})}<\infty.
\end{equation}
\end{proposition}

\begin{proof} 
One only needs to note that there is a non-trivial continuous group homomorphism $\Phi:G\to\mathbb{C}$ such that (\ref{32}) holds if and only if there is a non-trivial continuous group homomorphism $\Phi:H\to\mathbb{C}$ such that
\[
\sup\limits_{h\in\,H}\ \frac{|\Phi(h)|}{{\omega_H}(h){\omega_H}(h^{-1})}<\infty.
\]
So the conclusion follows from \cite[Theorem~3.1]{Zhang}.
\end{proof}

\begin{remark}\label{connected SIN}
According to \cite[Theorem~4.3]{GM}, if $G$ is a connected [SIN] group, then $G=V\times K$ for some (abelian) vector group $V$ and a compact group $K$. So Proposition~\ref{abelian/compact} is valid in particular for this kind of  group $G$.
\end{remark}


In fact, the necessity part of Proposition~\ref{abelian/compact} remains true for general [FC]$^-$ groups. To prove this we first consider a general [IN] group.

\begin{lemma}\label{IN case}
Let $G$ be an [IN] group, $\omega$ be a weight on $G$, and $U$ be an open set of $G$ with a compact closure and invariant under inner automorphisms of $G$. Suppose that there exists a continuous group homomorphism $\Phi:G\to\mathbb{C}$ non-trivial on $U$ and such that
$$
\sup_{t\in\, G}\,\frac{|\Phi(t)|}{\omega(t)\omega(t^{-1})}<\infty.
$$
Then there is a nontrivial continuous derivation from $ZL^1(G,\omega)$ into $\Lio$. Consequently, $Z\L1o$ is not weakly amenable.
\end{lemma}
\proof
Since $ZL^1(G,\omega)$ is the center of $\L1o$, $\Lio$ is a symmetric Banach $ZL^1(G,\omega)$-bimodule.  We construct a non-trivial continuous derivation $D$: $ZL^1(G,\omega)\to \Lio$. Whence this is done, it follows from the definition of weak amenability for a commutative Banach algebra given in \cite{BCD} that $ZL^1(G,\omega)$ is not weakly amenable. 

 We define $D$ as follows
\begin{equation}\label{1}
D(h)(t)=\int\limits_{U}\, \Phi(t^{-1}\xi)h(t^{-1}\xi)\,d\xi\quad (t\in G,\ h\in ZL^1(G,\omega)).
\end{equation}
First we note that $D$ is non-trivial. To see this we consider the function $h_{\Phi}=\chi_{{\phantom{}}_U}\overline{\Phi}$, where $\chi_{{\phantom{}}_U}$ is the characteristic function of $U$ and $\overline\Phi$ is the conjugate of $\Phi$. Since $\Phi$ is a group homomorphism and $U$ is invariant under $I(G)$ we have that $h_{\Phi}\in ZL^1(G,\omega)$. Moreover,
$$
D(h_{\Phi})(t)=\int\limits_U\, \Phi(t^{-1}\xi)h_{\Phi}(t^{-1}\xi)\,d\xi=\int\limits_{U\cap tU}\,|\Phi(t^{-1}\xi)|^2\,d\xi= \int\limits_{t^{-1}U\cap U}\,|\Phi(\xi)|^2\,d\xi.
$$
This shows that $D(h_{\Phi})(t)>0$ for $t$ in a neighbourhood of the identity $e$ of $G$ because  $|\Phi|^2 >0$ on some open subset of  $t^{-1}U\cap U$ when $t$ is near $e$. Hence, $D$ is non-trivial. Using the method of \cite[Theorem~2.2]{SZ1} one can show that the formula (\ref{1}) defines a bounded derivation even from the whole $L^1(G,\omega)$ into $\Lio$. So, it also defines a (non-trivial) continuous derivation from $ZL^1(G,\omega)$ into $\Lio$. 
\qed

We will need some elementary property of an [FC]$^-$ group, which we state as follows.

\begin{lemma}\label{fc}
Let $G$ be an [FC]$^-$ group. Then for every $x\in G$ there exists an open precompact neighbourhood $K_x$ of $x$ in $G$ that is invariant under inner automorphisms of $G$.
\end{lemma}

\begin{proof}
It is known that an [FC]$^-$ group $G$ belongs to [IN].  Let $B$ be a precompact open invariant neighbourhood of $e$ and let $C_x$ be the conjugacy class of $x$ (which is also precompact and invariant). Then $K_x = BC_x$ satisfies our requirement.
\end{proof}

\begin{proposition}\label{infc}
Let $G$ be a locally compact [FC]$^-$ group and $\omega$ be a weight on $G$. Suppose that there exists a non-trivial continuous group homomorphism $\Phi:G\to\mathbb{C}$ such that
\begin{equation}\label{Phi}
\sup_{t\in\, G}\,\frac{|\Phi(t)|}{\omega(t)\omega(t^{-1})}<\infty.
\end{equation}
Then $ZL^1(G,\omega)$ is not weakly amenable.
\end{proposition}

\begin{proof}
Since $\Phi$ is non-trivial, there exists $x\in G$ such that $\Phi(x)\ne0$. Applying Lemma \ref{fc}, we get an open neighbourhood $U=K_x$ of $x$ that is invariant under inner automorphisms and has compact closure. Therefore  Lemma~\ref{IN case} applies.
\end{proof}

We wonder whether the converse of Proposition~\ref{infc} remains true as in the commutative case. We raise it here as an open problem. We note that, in many cases, $Z\L1o$ is isomorphic to a weighted commutative hypergroup algebra. So our question links to the general open problem of characterizing weak amenability of (commutative) hypergroup algebras. In particular, it would be of great interest if one can characterize weak amenability of $Z^B\L1o$ for $G\in [FIA]^-_B$ and $B$ being a closed subgroup of $Aut(G)$ with $I(G)\subseteq B$ (see \cite{LM} for definitions).

\section{Central Beurling algebras on [FD]$^-$ groups}

In this section we  consider [FD]$^-$ groups and aim to establish some sufficient conditions for $ZL^1(G,\omega)$ to be weakly amenable.  We recall first that $G$ is an [FD]$^-$ group if and only if there exists a compact normal subgroup $K$ of $G$ such that the quotient $G/K$ is abelian.

The following structural result, which is \cite[Lemma~1]{Mosak3}  for $B=I(G)$, is crucial in the sequel.

\begin{lemma}\label{mosak}
Let $G$ be an [FD]$^-$ group and $K$ a compact normal subgroup of $G$ such that $G/K$ is abelian. Let $\omega\ge 1$ be a weight on $G$ satisfying 
$$\lim_{n\to\infty}(\omega(x^n))^{1/n}=1$$
 for all $x\in G$, and let $\hat{\omega}$
 be the induced weight on $G/K$ defined by 
 $$\hat{\omega}(xK)=\inf_{k\in K}\omega(xk)\quad (x\in G). $$
 Then $ZL^1(G,\omega)$ may be written as the closure of the linear span of a family of complemented ideals, each of which is isomorphic to a Beurling algebra of the form $L^1(S/K,\hat{\omega})$ for some open normal subgroup $S$ of $G$.
\end{lemma}


We need also the following well-known result.

\begin{lemma}\label{subalgebras}
Let $A$ be a commutative Banach algebra and $\{A_{\gamma}\}_{\gamma\in\Gamma}$ be a family of closed subalgebras of $A$ such that $A=\overline{\mathrm{lin}}\{A_{\gamma}\}_{\gamma\in\Gamma}$. If each $A_{\gamma}$ is weakly amenable, then so is $A$.
\end{lemma}



We note that $L^1(S/K,\hat{\omega})$ in Lemma~\ref{mosak} is a commutative Beurling algebra. So \cite[Theorem~3.1]{Zhang} applies for the weak amenability of it. This leads to the following result.

%


\begin{theorem}\label{G/K}
Let $G$ be an [FD]$^-$ group and $\omega\ge 1$ be a continuous weight on $G$ satisfying
\begin{equation}\label{33}
\sup_{n\in\mathbb{N}}\, \frac{n}{\omega(x^n)\omega(x^{-n})}=\infty\quad  ( x\in G).
\end{equation}
Then $ZL^1(G,\omega)$ is weakly amenable.
\end{theorem}

\begin{proof}
First we show that (\ref{33}) implies that $\lim_{n\to \infty} (\omega(x^n))^{1/n}=1$ for every $x\in G$. Since $\omega\ge 1$, it suffices to prove that
$$
\limsup_{n\to \infty} (\omega(x^n))^{1/n}\le 1\quad (x\in G).
$$
Fix $x\in G$ and let $\varepsilon>0$ be arbitrary. Because $\lim_{n\to\infty} n^{1/n}=1$, there exists $N_{\varepsilon}\in\mathbb{N}$ such that $n^{1/n}\le(1+\varepsilon)$ for every $n\ge N_{\varepsilon}$. Using the assumption (\ref{33}) and the inequality $\omega\ge 1$, we can find $n_{\eps}>N_{\eps}$ such that
$$
\omega(x^{n_{\eps}})\le\omega(x^{n_{\eps}})\omega(x^{-n_{\eps}})\le n_{\eps}=(n_{\eps}^{1/n_{\eps}})^{n_{\eps}}\le(1+\eps)^{n_{\eps}}.
$$
For any $m\in\mathbb{N}$ there exist $k\in\mathbb{N}\cup\{0\}$ and $0\le l<n_{\eps}$ such that $m=kn_{\eps}+l$. Using the weight inequality for $\omega$, we can make the following estimates
\begin{align*}
\omega(x^m)&=\omega(x^{kn_{\eps}+l})\le(\omega\left(x^{n_{\eps}})\right)^k\omega(x^l) \le(1+\eps)^{kn_{\eps}}\omega(x^l)\\&= \frac{(1+\eps)^m\omega(x^l)}{(1+\eps)^l}\le c_{\eps}(1+\eps)^m,
\end{align*}
where
$$
c_{\eps}=\sup_{0\le l<n_{\eps}}\frac{\omega(x^l)}{(1+\eps)^l}
$$
is a constant that does not depend on $m$. It follows that
$$
\limsup_{n\to \infty} (\omega(x^n))^{1/n}\le \limsup_{n\to \infty} (c_{\eps}(1+\eps)^n)^{1/n}=1+\eps.
$$
Since $\eps>0$ was arbitrary, we obtain that $\limsup_{n\to \infty} (\omega(x^n))^{1/n}\le1$, as desired.

So, the condition of Lemma~\ref{mosak} is satisfied. Then there exists a family of complemented ideals $\{J_{\gamma}\}_{\gamma\in\Gamma}$ of $ZL^1(G,\omega)$ such that $\overline{\mathrm{lin}}\{J_{\gamma}\}_{\gamma\in\Gamma}=ZL^1(G,\omega)$ and for each $\gamma\in\Gamma$ there exists an open subgroup $S_{\gamma}\supset K$ of $G$ for which $J_{\gamma}\simeq L^1(S_{\gamma}/K,\hat{\omega})$.  Let $\Phi: S_{\gamma}/K \to \mathbb{C}$ be a non-trivial continuous group homomorphism. Choose $t_{\gamma}\in S_{\gamma}/K$ so that $\Phi(t_{\gamma})\ne0$. Then
\begin{equation}\label{2}
\sup_{t\in\, S_{\gamma}/K}\,\frac{|\Phi(t)|}{\hat{\omega}(t)\hat{\omega}(t^{-1})}\ge \sup_{n\in\mathbb{N}}\, \frac{|\Phi(t_{\gamma}^n)|}{\hat{\omega}(t_{\gamma}^n)\hat{\omega}(t_{\gamma}^{-n})}=\sup_{n\in\mathbb{N}}\, \frac{n\,|\Phi(t_{\gamma})|}{\hat{\omega}(t_{\gamma}^n)\hat{\omega}(t_{\gamma}^{-n})}.
\end{equation}
Let $x_{\gamma}\in S_{\gamma}$ be a representative of $t_{\gamma}$, i.e. $x_{\gamma}K=t_{\gamma}$. We note that, for each $x\in G$, 
$$
{\hat\omega}(xK)=\inf_{k\in K}\,\omega(xk)\le\omega(x).
$$
 In particular, $\hat{\omega}(t_{\gamma}^n)\le \omega(x_{\gamma}^n)$ and $\hat{\omega}(t_{\gamma}^{-n})\le \omega(x_{\gamma}^{-n})$ ($n\in\mathbb{N}$).
Combining this with condition (\ref{33}) and (\ref{2}), we obtain
$$
\sup_{t\in S_{\gamma}/K}\frac{|\Phi(t)|}{\hat{\omega}(t)\hat{\omega}(t^{-1})}\ge \sup_{n\in\mathbb{N}}\,\frac{n\,|\Phi(t_{\gamma})|} {\hat{\omega}(t_{\gamma}^n)\hat{\omega} (t_{\gamma}^{-n})}\ge \sup_{n\in\mathbb{N}}\,\frac{n\,|\Phi(t_{\gamma})|}{\omega(x_{\gamma}^n) \omega(x_{\gamma}^{-n})}=\infty.
$$
According to \cite[Theorem~3.1]{Zhang}, this implies that $J_{\gamma}\simeq L^1(S_{\gamma}/K,\hat{\omega})$ is weakly amenable. Then Lemma~\ref{subalgebras} applies.
\end{proof}





We now apply Theorem~\ref{G/K} to compactly generated [FC]$^-$ groups, which are, in fact, [FD]$^-$ groups according to \cite[Theorem~3.20]{GM}.

Let $G$ be a compactly generated locally compact group. Then there is an open symmetric neighbourhood $U$ of the identity in $G$ with compact closure and satisfying $G=\bigcup_{n=1}^{\infty} U^n$. Following~\cite{Mosak3}, we consider the length function $|\cdot|: G\to\mathbb{N}$ defined by
$$
|x|=\min\{n\in\mathbb{N}: x\in U^n\}\quad (x\in G).
$$
It is readily checked that $|x| \geq 1$ ($x\in G$), and for every $\alpha\geq 0$  the corresponding polynomial weight $\omega_{\alpha}(x)=(1+|x|)^{\alpha}$ ($x\in G$) is, indeed, an upper semicontinuous weight on $G$. As addressed in the introduction section, $\omega_\al$ is equivalent to a continuous weight. 

\begin{theorem}\label{polynomial}
Let $G$ be a compactly generated non-compact [FC]$^-$ group and $\omega_{\alpha}$ be the weight on $G$ defined as above. Then $ZL^1(G,\omega_{\alpha})$ is weakly amenable if and only if $0\leq\alpha<1/2$. 
\end{theorem}

\begin{proof}
From the definition of the length function we have $|x^{-1}|=|x|$ and $|x^n|\le n|x|$ ($x\in G$, $n\in\mathbb{N}$). If $0\leq\alpha<1/2$, then
\begin{align*}
\sup_{n\in\mathbb{N}}\, \frac{n}{\omega_{\alpha}(x^n)\omega_{\alpha}(x^{-n})}&= \sup_{n\in\mathbb{N}}\, \frac{n}{(1+|x^n|)^{\alpha}(1+|x^{-n}|)^{\alpha}}\\ &\ge \sup_{n\in\mathbb{N}}\, \frac{n}{(1+n|x|)^{2\alpha}}=\infty\qquad (x\in G).
\end{align*}
This is still true if $\omega_\al$ is replaced by a continuous equivalent weight. Therefore, $ZL^1(G,\omega_{\alpha})$ is weakly amenable by Theorem~\ref{G/K}.

To prove the converse we let $K$ be a compact subgroup of $G$ such that $G/K$ is abelian. The quotient group $H=G/K$ is clearly still compactly generated. By the structure theorem for compactly generated locally compact abelian groups \cite[Theorem~II.9.8]{H-R}, $H$ is topologically isomorphic to $\R^m\times\bZ^n\times F$ for some integers $m$ and $n$ and some compact abelian group $F$. Since $G$ is not compact, neither is $H$. Then either $\R$ or $\bZ$ is a quotient group of $H$. Thus there is a non-trivial continuous group homomorphism $\phi$: $H\to \R$. Then $\Phi = \phi\circ q$: $G \to \R$ is a non-trivial continuous group homomorphism, where $q$: $G\to H$ is the quotient map. If $\alpha \geq 1/2$, this $\Phi$ satisfies the inequality~(\ref{Phi}) with $\omega = \omega_\alpha$. In fact, for $x\in G$ there is a smallest $k\in \bN$ such that $x\in U^k$. We have $|x| = k$ and 
\[
|\Phi(x)| \leq c_0 k = c_0 |x|,
\]
where $c_0 = \sup_{g\in U} |\Phi(g)|$ which is finite since $\overline U$ is compact. This leads to inequality~(\ref{Phi}) for $\omega = \omega_\al$ (and also for any continuous $\omega$ equivalent to $\omega_\al$) since $\al \geq 1/2$.
Hence $L^1(G, \omega_\al)$ is not weakly amenable due to Proposition~\ref{infc}.
\end{proof}

\begin{remark}
Consider again the general [FD]$^-$ group $G$. Let $K$ be a compact normal subgroup of it such that $G/K$ is commutative. If there is a continuous non-trivial group homomorphism $\Phi$: $G \to \mathbb C$ such that (\ref{Phi}) holds, then $Z\L1o$ is not weakly amenable from Proposition~\ref{infc}. If there is no such $\Phi$, then there is no such $\Phi$ for $G/K$ with the weight $\hat\omega$. Then $L^1(G/K, \hat\omega)$ is weakly amenable due to Theorem~\cite[Theorem~3.1]{Zhang}. We want to know whether $L^1(S/K,\hat\omega)$ is weakly amenable for any open subgroup $S$ of $G$ containing $K$. If this is true we will then obtain a characterization for the weak amenability of Beurling algebras on an [FD]$^-$ group. We note that $S/K$ is an open subgroup of $G/K$. However, in general weak amenability of a Beurling algebra does not pass to the Beurling algebra on a subgroup (see section 5 of \cite{SZ1}).
\end{remark}

The situation is simple when $G/K$ is isomorphic to $\mathbb R$ or $\mathbb Z$.

\begin{proposition}\label{rz}
Suppose that $G$ is a locally compact group and has a compact normal subgroup $K$ such that $G/K\simeq\mathbb{R}$ or $G/K\simeq\mathbb{Z}$. Let $\omega\ge 1$ be a weight on $G$. Then $ZL^1(G,\omega)$ is weakly amenable if and only if there is no non-trivial continuous group homomorphism $\Phi:~G\to\mathbb{C}$ such that
\begin{equation}\label{34}
\sup_{t\in\, G}\,\frac{|\Phi(t)|}{\omega(t)\omega(t^{-1})}<\infty.
\end{equation}
\end{proposition}

\begin{proof}
It suffices to show the sufficiency. If there is no non-trivial continuous group homomorphism $\Phi:G\to\mathbb{C}$ for which (\ref{34}) holds, then, as is known,  $L^1(G/K,\hat{\omega})$ is weakly amenable.
This in turn implies that 
$$
\sup_{n\in\mathbb{N}}\,\frac n{\hat{\omega}(t^n)\hat{\omega}(t^{-n})}=\infty\quad (t\in G/K)
$$
due to \cite[Corollary~3.7]{Zhang}. Since $\omega(x) \leq c\hat\omega(xK)$ for $x\in G$, where $$c=\max\{\omega(k): k\in K\},$$
the last condition leads to
$$
\sup_{n\in\mathbb{N}}\,\frac n{\omega(x^n)\omega(x^{-n})}=\infty\quad (x\in G).
$$
Then, applying Theorem~\ref{G/K}, we conclude that $ZL^1(G,\omega)$ is weakly amenable.
\end{proof}

The authors are grateful to the referees for their thoughtful comments and suggestions.

\bibliographystyle{plain}


\end{document}